\documentclass[12pt]{amsart}
\usepackage{amsmath,amsfonts,amssymb}

\def\le{\leqslant}
\def\ge{\geqslant}

\renewcommand{\mod}{\mathop{\rm{mod}}}

\numberwithin{equation}{section}

\newtheorem{cor}{COROLLARY}[section]

\newtheorem{lemma}[cor]{LEMMA}

\theoremstyle{definition}

\newtheorem{theorem}{THEOREM}

\begin{document}
\title{\bf On a theorem of Bredihin and Linnik}
\author[Friedlander]{J.B. Friedlander$^*$}
\thanks{$^*$\ Supported in part by NSERC grant A5123}
\author[Iwaniec]{H. Iwaniec$^{**}$}
\thanks{$^{**}$\ Supported in part by NSF grant DMS-1406981}

\maketitle
\dedicatory \quad\quad\quad\quad\quad {\sl Dedicated to the memory of Yu. V. Linnik.}

\medskip

{\bf Abstract:} We give a new proof of a theorem of B. M. Bredihin which was 
originally proved by extending Linnik's solution, via his dispersion method, 
of a problem of Hardy and Littlewood. 
\footnote{keywords: primes, dispersion, Bombieri-Vinogradov theorem} 
\section{\bf Introduction}

Among the many beautiful consequences of Linnik's dispersion method is an asymptotic formula for the number of solutions to the equation

$$
p= a^2 + b^2 +1 
$$ 
in primes $p\le x$ and integers $a$ and $b$. 
This result of 1965, due to Bredihin [B] was a follow-up to Linnik's celebrated work on the Hardy-Littlewood problem, cf. Chapter 7 of [L]. The involved arguments are lengthy and complicated, though very inventive. Due to much progress over the intervening years, much shorter arguments can now be put forward. This of course does not mean that they are shorter ab-initio. Our purpose here is to illustrate how these arguments can be applied.

\begin{theorem}\label{1} 
Let $S(x)$ be the number of solutions to 
\begin{equation}\label{eq:1.1}
p= a^2 + b^2 +1  
\end{equation} 
in integers $a$ and $b$ and primes $p \equiv 3 (\mod 8), \, p \le x$. We have 
\begin{equation}\label{eq:1.2}
S(x) = c\, 
%%\frac{\pi}{4}
%%L(1,\chi) 
\frac{x}{\log x} 
+ O\Bigl(x\Bigl( \frac{\log\log x}{\log x}\Bigr)^2\Bigr) ,
\end{equation} 
where the constant $c$ is given by 
\begin{equation}\label{eq:1.3}
c=\frac{\pi}{2}\prod_p\Bigl(1+\frac{\chi (p)}{p(p-1)}\Bigr) , 
\end{equation} 
with $\chi$ being the Dirichlet character of conductor $4$.

\end{theorem}

\medskip 
\noindent The other reduced residue classes modulo $8$ can be covered by 
essentially the same arguments but we do not treat them. 

Note that the theorem shows that the integers $p-1$ tend to have 
about as many representations as the sum of two squares as does a typical 
integer $n$. Recall also that, if the number of representable $p-1$ is counted 
without multiplicity in $a$ and $b$, then the order of magnitude is given by  
$x/(\log x)^{3/2}$ by a theorem of the second-named author [I].

\section {\bf Dirichlet divisor switching}
 
Let 
$\lambda = 1 * \chi$  that is
\begin{equation}\label{eq:2.1}
\lambda (n) =\sum_{ab=n} \chi (a)
\end{equation}  
This is similiar in many respects to the divisor function $\tau(n)$. The number of representations of $n$ as the sum of two squares is equal to $4\lambda (n)$. If $n\equiv 1 (\mod 4)$ then, in \eqref{eq:2.1}, $\chi(a)$ can be replaced by $\chi (b)$; therefore we can write 
\begin{equation}\label{eq:2.2}
\lambda (n) =\sum_{\substack{a|n\\a\le y}} \chi (a) +\sum_{\substack{b|n\\b<n/y}} \chi (b)
\end{equation}  
for any $y>0$. We can refine this partition by integrating over $y$ against a smooth weight function. Let $w(t)$ be a smooth function supported on $1\le t\le 2$ such that 
\begin{equation}\label{eq:2.3}
\int_0^\infty w(t)t^{-1} dt =1 . 
\end{equation}  

Let $Y\ge 1$, multiply \eqref{eq:2.2} by $w(y/Y)$ and integrate with the measure 
$y^{-1} dy$, getting 
\begin{equation}\label{eq:2.4}
\lambda (n) =\int_0^\infty \Bigl[w\bigl(\frac{y}{Y}\bigr)+w\bigl(\frac{n}{yY}\bigr)\Bigr] \Bigl(\sum_{\substack{b|n\\ b<y}}\chi(b)\Bigr)\frac {dy}{y}\ .
\end{equation}  
Note that if $X<n\le 2X$ we can choose $Y=\sqrt X$ so the integration in 
\eqref{eq:2.4} runs over the segment $\frac12\sqrt X < y < 2\sqrt X$.

\section {\bf Primes in arithmetic progressions}

The key input which greatly streamlines the proof is the main result of 
[BFI] which gives asymptotics of Bombieri-Vinogradov type for the distribution 
of primes in arithmetic progressions and which treats moduli of the progression 
which go beyond the range of that which can be sucessfully handled even on the 
assumption of the Generalized Riemann Hypothesis. 

We state this restricted to a range somewhat lesser than that in [BFI], which is however 
sufficient for our needs 
and is conveniently recorded as Theorem 2.2.1 of [FI]. 

\begin{equation}\label{eq:3.1}
\sum_{\substack{q\le Q\\(q,a)=1}}\Bigl|\pi(x;q,a)-\frac{\pi(x)}{\varphi(q)}\Bigr|\ll 
x\Bigl(\frac{\log\log x}{\log x}\Bigr)^2 
\end{equation} 
for $Q=\sqrt x (\log x)^A$ with any $a\neq 0$, $A\ge 0$, $x \ge 3$, the implied 
constant depending only on $a$ and $A$. 

%%For the proof (of a stronger result) 
%%see [BFI]. The arguments there allow us to generalize \eqref{eq:3.1} 
%%easily as follows:

We actually require a slightly modified form of \eqref{eq:3.1} which 
follows from it 
in two easy steps. In the first place we have 
\begin{equation}\label{eq:3.2}
\sum_{\substack{q\le Q\\(q,a)=1\\(q,k)=1}}
\Bigl|\sum_{\substack{p\le x\\p\equiv a (\mod q)\\p\equiv\ell (\mod k)}}1 
-\frac{\pi(x)}{\varphi(qk)}\Bigr|\ll 
x\Bigl(\frac{\log\log x}{\log x}\Bigr)^2 
\end{equation} 
for $Q=\sqrt x (\log x)^A$ with any $a\neq 0$, $k\ge 1$, $(\ell, k)=1$, 
$A \ge 0$, $x \ge 3$, the implied 
constant depending only on $a$,$k$ and $A$. To this end one merely splits the 
indexed variables into classes modulo $k$, which is harmless for $k$ fixed. 

In the second step we modify \eqref{eq:3.2} to a counting of primes with 
smooth weight. 
\begin{lemma}
Let $f(t)$ be a smooth function supported on $1\le t\le2$. We have 
\begin{equation}\label{eq:3.3}
\sum_{\substack{q\le Q\\(q,a)=1\\(q,k)=1}}\Bigl|\sum_{\substack{p\le x\\p\equiv a (\mod q)\\p\equiv\ell (\mod k)}}f\bigl(\frac{p}{X}\bigr) -\frac{1}{\varphi(qk)}\sum_pf\bigl(\frac{p}{X}\bigr)\Bigr|\ll 
x\Bigl(\frac{\log\log x}{\log x}\Bigr)^2 
\end{equation} 
for $Q=\sqrt x (\log x)^A$ with any $a\neq 0$, $k\ge 1$, $(\ell, k)=1$, 
$A \ge 0$, $x \ge 3$, the implied 
constant depending only on $a$,$k$, $A$ and $f$.
\end{lemma} 

\begin{proof} 
%%(sketch)
We write 
$$
f\bigl(\frac{p}{X}\bigr)= -\int^\infty_{p/X}f'(t)dt .
$$
Given $1\le t\le 2$ this implies $p\le tX$. Applying \eqref{eq:3.2} with 
$x=tX$ and integrating the result over $t$, we derive \eqref{eq:3.3}.
\end{proof}

\section {\bf Proof of the theorem}

We have 
\begin{equation}\label{eq:4.1}
S(x) = 4\sum_{\substack{p\le x\\p\equiv 3(\mod 8)}}\lambda\bigl(\frac{p-1}{2}\bigr) .
\end{equation} 
We are going to evaluate 
\begin{equation}\label{eq:4.2}
T(X) = S(2X)-S(X) = 4\sum_{\substack{X<p\le 2X\\p\equiv 3(\mod 8)}}
\lambda\bigl(\frac{p-1}{2}\bigr) 
\end{equation} 
for every $X\ge 3$. Applying \eqref{eq:2.4} we write 
$$
T(X)= 4\int\sum_{b<y}\chi(b)\sum_{\substack{X<p\le 2X\\p\equiv 1 (\mod b)\\p\equiv 3(\mod 8)}} 
\Bigl[w\bigl(\frac{y}{Y}\bigr)+w\bigl(\frac{p-1}{2yY}\bigr)\Bigr]\frac{dy}{y}
$$ 
where we choose $Y=\sqrt X$. Here we can replace $w((p-1)/2yY)$ by $w(p/2yY)$ 
up to an error term $O(1/yY)$ which contributes to $T(X)$ a bounded amount: 
$$
T(X)= 4\int\sum_{b<y}\chi(b)\sum_{\substack{X<p\le 2X\\p\equiv 1 (\mod b)\\p\equiv 3(\mod 8)}} 
\Bigl[w\bigl(\frac{y}{Y}\bigr)+w\bigl(\frac{p}{2yY}\bigr)\Bigr]\frac{dy}{y} 
+O(1)\, .
$$ 
Note that the integration runs over the segment $\frac14 \sqrt X <y<2\sqrt X$. 
Now we can apply \eqref{eq:3.2} for the first term and \eqref{eq:3.3} for the 
second term with $q=b$, $k=8$, $\ell =3$, getting 
$$
T(X)= \int\sum_{b<y}\frac{\chi(b)}{\varphi(b)}\sum_{X<p\le 2X} 
\Bigl[w\bigl(\frac{y}{Y}\bigr)+w\bigl(\frac{p}{2yY}\bigr)\Bigr]\frac{dy}{y} 
+O\Bigl(X\Bigl(\frac{\log\log X}{\log X}\Bigr)^2\Bigr)\, .
$$ 
Next, we replace the sum over $b<y$ by the complete series
\begin{equation}\label{eq:4.3}
c_1 =\sum_b\frac{\chi(b)}{\varphi(b)}=L(1,\chi)
\prod_p\Bigl(1+\frac{\chi(p)}{p(p-1)}\Bigr) 
\end{equation} 
up to an error term $O(1/y)$ which contributes to $T(X)$ at 
most $O(\sqrt X/\log X)$. Now the free integration over $y$ yields 
(see \eqref{eq:2.3}) 
$$
\int \Bigl[w\bigl(\frac{y}{Y}\bigr)+w\bigl(\frac{p}{2yY}\bigr)\Bigr]\frac{dy}{y} 
=2\, . 
$$ 
Therefore, 
$$
T(X)= 2c_1\bigl(\pi(2X)-\pi(X)\bigr) +O\Bigl(X\Bigl(\frac{\log\log X}{\log X}\Bigr)^2\Bigr)\, .
$$
Summing this over $X=2^{-n}x$, $n=1,2,3,\ldots$, we derive \eqref{eq:1.2}, thus completing the proof of theorem 1. \endproof

\medskip 
Department of Mathematics, University of Toronto

Toronto, Ontario M5S 2E4, Canada 

\medskip

Department of Mathematics, Rutgers University

Piscataway, NJ 08903, USA

\end{document}